\documentclass[a4paper,11pt]{article}   
\usepackage{lineno,hyperref}
\usepackage[T2A]{fontenc}
\usepackage[utf8]{inputenc}
\usepackage[english]{babel}
\usepackage{ifthen}
\usepackage{amssymb,amsmath}
\usepackage[usenames]{color}
\date{}

\newenvironment{proof}[1][\hspace{-1.0ex}]%
{\par\addvspace{1mm}{\it Proof\hspace{1.0ex}{#1}.} }%
{\quad$\square$\par\addvspace{1mm}}

    \newif\ifNoRemark
    \def\addtheorem#1#2#3#4{ 
    \ifthenelse{\expandafter\isundefined\csname the#2\endcsname}{\newcounter{#2}}{}
    \newenvironment{#1}[1][\global\NoRemarktrue]
     {\par\addvspace{2mm}\noindent 
       \refstepcounter{#2}{\bf #3~\csname the#2\endcsname
      \vphantom{##1}\ifNoRemark.\ \else\ (##1).\fi}\begingroup #4}%
     {\endgroup\par\addvspace{1mm}\global\NoRemarkfalse}
    \expandafter\newcommand\csname b#1\endcsname{\begin{#1}}
    \expandafter\newcommand\csname e#1\endcsname{\end{#1}}
    }

\newtheorem{theorem}{Theorem}
\newtheorem{claim}{Proposition}

\newtheorem{corollary}{Corollary}

\newcommand{\supp}{{\rm supp}}

\title{Combinatorial designs, difference sets and bent functions as
perfect colorings of graphs and multigraphs}
\author{V.N.Potapov and S.V.Avgustinovich \\
 quasigroup349@gmail.com, avgust@math.nsc.ru}

\begin{document}

\maketitle

\begin{abstract}
It is proved that  1) the indicator function of some onefold or
multifold independent set in a regular graph is a perfect coloring
if and only if the  set attain the Delsarte--Hoffman bound; 2) each
transversal in a uniform regular hypergraph is an independent set
attaining the Delsarte--Hoffman bound in the vertex adjacency
multigraph of this hypergraph; 3) combinatorial designs with
parameters
 $t$-$(v,k,\lambda)$ and similar
$q$-designs, difference sets, Hadamard matrices, and bent functions
are equivalent to perfect colorings of special graphs and
multigraphs, in particular, it is true in the cases of the Johnson
graphs $J(n,k)$ for $(k-1)$-$(v,k,\lambda)$ designs and the
Grassmann graphs $J_2(n,2)$ for bent functions.


MSC 05C15; 05C50; 05B05

\end{abstract}

\section*{Introduction}

Let $G=(V,E)$ be a graph. A function mapping from the vertex set
$V(G)$ to a finite set $I$ of colors is called a coloring of the
graph. A coloring $f$ is called {\it perfect} if each collection  of
vertices adjacent to vertices of the same color have identical color
composition. It can be formulated as follows: for each color $i\in
I$ and each pair of vertices  $x,y\in V(G)$ we require that
$f(x)=f(y)$ follows $|f^{-1}(i)\cap S(x)|= |f^{-1}(i)\cap S(y)|$,
where $S(x)=\{z\in V(G) : \{z,x\}\in E(G)\}$ is the set of vertices
adjacent to $x$. Also the  terms 'equitable partition'\ and
'partition design'\ are used  for the set $\{f^{-1}(i) : i\in I\}$.
This set is a partition  of vertices of $G$ according to the colors
of a perfect coloring. The matrix $P=(p_{ij})$ of size $|I|\times
|I|$ is called a quotient matrix if each its entry $p_{ij}$ is equal
to the number of  vertices of color $j$  adjacent to each vertex of
color $i$.

Perfect $2$-colorings turn out to be solutions to extremal problems
on graphs. The definition of $1$-perfect code in an $r$-regular
graph imply directly that the indicator function of the code is a
perfect coloring with quotient matrix  $\begin{pmatrix}
0 & r \\
1 & r-1
\end{pmatrix}$. It is proved in \cite{FdF2} that all unbalanced Boolean functions
reaching  Fon-Der-Flaass's bound of correlation immunity are perfect
$2$-colorings. All Boolean functions attaining the
Bierbrauer--Friedman bound for orthogonal arrays are also perfect
$2$-colorings (see \cite{Pot}). We proved another similar result. In
Theorem \ref{thDH} we establish that if an onefold or multifold
independent set in a regular graph attains the Delsarte--Hoffman
bound then the indicator function of this set is a perfect
$2$-coloring.

We suggest characterizations for various well-known combinatorial
configurations in terms of perfect colorings of simple graphs,
hypergraphs and multigraphs. In Proposition \ref{exerDH} we prove
that each transversal of a uniform regular hypergraph  is an
independent set in some multigraph attaining the Delsarte--Hoffman
bound. Consequently, it is equivalent to a perfect $2$-coloring of
the multigraph. In Section 4 we show that combinatorial designs and
$q$-designs are equivalent to perfect $2$-colorings with certain
parameters of the Johnson graphs and the Grassmann graphs
respectively or multigraphs obtained from these graphs. In
particular, the Hadamard matrices turn out to be equivalent to
perfect colorings of some hypergraph. In Section 6 we prove that
strongly regular graphs and partial difference sets in abelian
groups are equivalent to perfect $2$-colorings of the line graph of
the complete graph $K_n$. At last, we prove that Boolean bent
functions are equivalent to perfect $4$-colorings of $J_2(n,2)$ with
certain quotient matrix (see Theorem \ref{avg}).

\section{Perfect Colorings of Multigraphs and Hypergraphs}

{\it Multigraph} is a generalization of the  notion of graph by
allowing multiple edges and loops. Let us enumerate vertices of a
multigraph on $n$ vertices by numbers from $1$ to $n$. A square
matrix $A=(a_{ij})$ of size $n\times n$ is called an adjacency
matrix of a multigraph if each entry $a_{ij}$ is equal to the number
of edges  joining $i$th and $j$th vertices.  Each adjacency matrix
of multigraph is symmetric; in the case of a simple graph it
consists only of the numbers $0$ and $1$, moreover, it has only
zeros at the main diagonal. The adjacency matrix of a bipartite
graph has a block structure. Eigenvalues and eigenvectors of
adjacency matrices of a multigraph are called eigenvalues and
eigenvectors of the multigraphs.

Let $G$ be a bipartite graph with parts $V$ and $U$. Consider the
multigraph $\mathcal{M}_{12}(G)$ having $V$ as the set of vertices
and vertices $v_i,v_j\in V$ are joined by $m_{ij}$ edges if there
exist  $m_{ij}$ distinct vertices of $U$ such that each of them is
adjacent to $v_i$ and $v_j$. For convenience we assume that every
vertex of $\mathcal{M}_{12}(G)$ is incident to $d$ loops where $d$
is the degree of the vertex in $G$. If all edges and loops in
$\mathcal{M}_{12}(G)$ have the same multiplicity then in all
statements of this section instead of $\mathcal{M}_{12}(G)$ we may
consider the simple graph obtained from $\mathcal{M}_{12}(G)$ by
removing  loops and multiple edges.

The adjacency matrix of the bipartite graph $G$ can be represented
as $M=\begin{pmatrix}
0 & Y \\
Y^* & 0
\end{pmatrix}$, where $Y$ is a matrix of size $k_1\times k_2$ and $k_i$
is the cardinality of the $i$th part  of $G$. Then the adjacency
matrix of $\mathcal{M}_{12}(G)$ is $YY^*$.

Some part of the following proposition for bipartite graphs is
proved in \cite{AM}.

\begin{claim}\label{eigenfun}

$1.$ The eigenvalues of  $\mathcal{M}_{12}(G)$ are nonnegative.

$2.$ The restriction of every eigenfunction of $G$ with  eigenvalue
$\theta$ to the first part of the graph is an eigenfunction of
$\mathcal{M}_{12}(G)$ with the eigenvalue $\theta^2$.

$3.$ Every eigenfunction of $\mathcal{M}_{12}(G)$ with a positive
eigenvalue $\theta$ can be extent to an eigenfunction of $G$ with
the eigenvalue $\sqrt{\theta}$.

$4.$ Every eigenfunction of $\mathcal{M}_{12}(G)$ with eigenvalue
$0$ can be extended to an eigenfunction of  $G$ with the same
eigenvalue.
\end{claim}
\begin{proof}
1. It follows from $YY^*$ is a nonnegative semidefinite matrix for
every $Y$.

2.   Let $(h,g)$ be an eigenvector of $M$ with  eigenvalue $\theta$,
where $h$ is a vector of length $k_1$ and $g$ is a vector of length
$k_2$. Then $Yg=\theta h$ and $Y^*h=\theta g$. Therefore,
$YY^*h=\theta Yg=\theta^2h$.

3. Suppose $Y^*h=\theta h$. Consider a vector $f=(\sqrt{\theta}h,
Y^*h)$. Then the equation $Mf=\sqrt{\theta}f$ is equivalent to the
pair of equations $YY^*h={\theta}h$ and $\sqrt{\theta}Y^*h=
\sqrt{\theta}Y^*h$.

4. Suppose that $YY^*h=\bar0$. Then $(Y^*h,Y^*h)=(YY^*h,h)=0$, which
implies that $Y^*h=\bar0$. Then $Mf=\bar0$, where $f=(h,\bar 0)$.
\end{proof}

The definition of a perfect coloring is naturally generalized to
multigraphs and directed multigraphs. We mean that a vertex in a
directed multigraph is adjacent to another  vertex if there exists
an arc (directed edge) from the first vertex to the second one. The
definition of a quotient matrix for a perfect coloring of a directed
multigraph is similar to the definition of a quotient matrix for
graphs. Notice that we can consider each quotient matrix of a
perfect coloring as an adjacency matrix of some directed multigraph.
The number of vertices in this directed multigraph is equal to the
number of colors in the perfect coloring.

Given a coloring $f: V(G) \rightarrow \{1,\ldots,k\}$ of $G$. Define
the matrix $F$ of size $n\times k$ as follows: every column  $F_i$
is the indicator function of color $i$. More formally we have
equation:
\[
\chi_{f^{-1}(i)}(x)=\left\{
\begin{array}{ll}
1 & \mbox{if}\ f(x)=i,\\
0 & \mbox{if}\ f(x)\ne i.
\end{array}
\right.
\]

Notice that each row of $F$ contains only one $1$. Henceforth, we
represent a function mapping from the set of vertices by a column
vector.

In the case of simple graphs it is well known the linear algebraic
criterion for perfect colorings. Below we prove this criterion for
directed multigraphs in a similar way.

\begin{claim}\label{s:perfcol_criteria}
Let $M$ be the adjacency matrix of a directed multigraph $G$. The
function $f$ is a perfect $k$-coloring of $G$ with the quotient
matrix $S$ if and only if  $MF=FS$.
\end{claim}
\begin{proof}
Let $j=1,\dots,k$ and $v\in V(G)$. Consider entries labeled by $(v
j)$ on the both sides of the equality  $MF=FS$. $(MF)_{v j}$ is the
number of  vertices of the color $j$  adjacent  to $v$. In the
right-hand side the row $F_v$ contains $1$ in the position
corresponding to the color of $v$. This implies that $(FS)_{v j}$ is
the number of vertices of the color $j$ adjacent to $v$.
\end{proof}

The corollaries of Proposition \ref{s:perfcol_criteria} listed below
can be found in \cite{FdF1} in the case of simple graphs. Their
proofs in the case of directed multigraphs are similar.

\begin{corollary}\label{p:perf_color_properties}
Let $u$ be an eigenvector of $S$.    Then $Fu$  is an eigenvector of
$M$ with the same eigenvalue.
\end{corollary}

\begin{corollary}\label{corol5}
A function $f:V(G)\rightarrow \{0,1\}$ is a perfect coloring of an
$r$-regular multigraph $G$ with  quotient matrix $\begin{pmatrix}
r-b & b \\
c & r-c
\end{pmatrix}$ if and only if $
h(x)=\left\{
\begin{array}{ll}
\frac{b}{b+c} & f(x)=0;\\
-\frac{c}{b+c} & f(x)=1.
\end{array}
\right. $ is an eigenfunction of  $G$ with the eigenvalue $r-b-c$.
\end{corollary}

\begin{claim}\label{bipartity}
If $f:V\cup U \rightarrow \{1,\dots,k\}$ is a perfect coloring of a
bipartite graph     $G$  with parts $V$ and  $U$, then $f|_V$ is a
perfect coloring of
    $\mathcal{M}_{12}(G)$.
\end{claim}
\begin{proof}
The adjacency matrix of $G$ can be represented as $M=\begin{pmatrix}
0 & Y \\
Y^* & 0
\end{pmatrix}$. Let $S=\begin{pmatrix}
0 & S_1 \\
S_2 & 0
\end{pmatrix}$ be a quotient matrix of the perfect coloring. Proposition
 \ref{s:perfcol_criteria} implies equality $MF=FS$ that is
 equivalent to the pair of equalities $YF_2=F_1S_1$ and $Y^*F_1=F_2S_2$, where
  $F=\begin{pmatrix}
F_1 & 0 \\
0 & F_2
\end{pmatrix}$, and matrices $F_1$ and $F_2$ correspond to
colorings of two parts of  $G$. We obtain
$YY^*F_1=YF_2S_2=F_1S_1S_2$. Observe that $YY^*$ is the adjacent
matrix of $\mathcal{M}_{12}(G)$ and  $S_1S_2$ is the quotient matrix
of the perfect coloring  $f|_V$ by Proposition
\ref{s:perfcol_criteria}.
\end{proof}

The converse is false in general. A perfect coloring of
$\mathcal{M}_{12}(G)$ do not need to be a restriction of any perfect
coloring of  $G$. In particular, we can find some examples in the
case that $\mathcal{M}_{12}(G)$ is the complete graph.

A {\it hypergraph} is another generalization of a graph. Here we
consider
 unordered collections of various numbers of vertices as  hyperedges
 of the hypergraph.
 A hypergraph is called  {\it $k$-uniform} whenever each hyperedge
 consists of $k$ vertices. A $(0,1)$-matrix $Y=(y_{ij})$ of size  $n\times
 m$, where $n$ is the number of vertices and $m$ is the number of
 hyperedges,
 is called an {\it incidence matrix} of the hypergraph if
\[
y_{ij}= \left\{
\begin{array}{l}
1 \mbox{, if}\ v_i\in e_j \in E(G);\\
0\mbox{, otherwise}.
\end{array}
\right.
\]

For any hypergraph $G$ we define the bipartite graph
$\mathcal{D}(G)$. The first part of $\mathcal{D}(G)$ consists of all
vertices of $G$ and the second one consists of all hyperedges of
$G$. A vertex $v\in V(G)$ from the first part  is adjacent to a
vertex $e\in E(G)$ from the second part whenever $v\in e$. It is
easy to see that the adjacency matrix of $\mathcal{D}(G)$ can be
represented as $\begin{pmatrix}
0 & Y \\
Y^*& 0
\end{pmatrix}$, where $Y$ is the incidence matrix of  $G$.

A coloring of  $G$ is called  {\it perfect} if it induces a perfect
coloring of  $\mathcal{D}(G)$. In other words, a coloring of a
hypergraph is perfect whenever two vertices of the same color are
incident to the equal numbers of hyperedges with any fixed color
composition of vertices.

Consider the adjacency matrix $M$ of the multigraph
$\mathcal{M}_{12}(\mathcal{D}(G))$. It is clear that the entry
$m_{ij}$ of $M$ is equal to the number of hyperedges of $G$
containing $i$th and $j$th vertices at the same time. Then $M=YY^*$,
where $Y$ is the incidence matrix of  $G$.

By Proposition \ref{bipartity} we obtain
\begin{corollary}
Each perfect coloring of the hypergraph $G$  is a perfect coloring
of $\mathcal{M}_{12}(\mathcal{D}(G))$.
\end{corollary}

Let $G$ be a $k$-uniform hypergraph. The multigraph $\mathcal{E}(G)$
is called a {\it line graph} of $G$ if vertices of $\mathcal{E}(G)$
are edges of $G$ and two vertices $v,u\in \mathcal{E}(G)$ are joined
by $l$ edges whenever $u\cap v$ consists of $l$ vertices of $G$.

By the definition, the adjacency matrix of  $\mathcal{E}(G)$ is
$Y^*Y-kI$, where $Y$ is the incidence matrix of $G$. If $G$ is a
simple graph then $\mathcal{E}(G)$ is also a simple graph.

A coloring $f'$ of $\mathcal{E}(G)$ is   {\it induced} by a coloring
$f$ of $G$ whenever  we use a multiset of colors of vertices from
$e\in E(G)$  as a  color of  $e\in \mathcal{E}(G)$, i.\,e.,
$f'(e)=\{f(v_1),f(v_2),\dots,f(v_k)\}$, where
$e=\{v_1,v_2,\dots,v_k\}$.

\begin{claim}\label{cl:edge_coloring}
Each coloring induced by a perfect coloring of a hypergraph $G$ is
perfect.
\end{claim}
\begin{proof}
Let $f$ be a perfect $t$-coloring of $G$. Consider some hyperedge
$e$ of the color  $\{i_1,\dots,i_t\}$, i.\,e., $e$ contains  $i_j$
vertices of color  $j$.  By the definition of the  perfect coloring
of hypergraphs, each vertex of the color $j$ belongs to a known
number of hyperedges of each fixed color composition. The multiset
of such color compositions of all vertices in $e$ is the same for
hyperedges with the same color. Then the collection  of vertices
adjacent to vertices of $\mathcal{E}(G)$ with the same color have
identical color composition. It is possible that some hyperedge $e'$
intersects $e$ by $s>1$ vertices, then the color composition of $e'$
will be counted $s$ times. It corresponds to the fact that the
vertices $e$ and $e'$ are joined by $s$ edges  in $\mathcal{E}(G)$.
\end{proof}

\section{Transversals of Hypergraphs}

A {\it transversal} of a hypergraph is a set of vertices that covers
each hyperedge one time, i.\,e., the intersection between the
transversal and  each hyperedge consists of  one vertex. An
$\ell$-{\it fold transversal} of a hypergraph is a set of vertices
which cover each edge $\ell$ times.

In particular,  in a bipartite graph each part is a transversal of
the graph. It is easy to see that the indicator function of each
multifold transversal is a perfect $2$-coloring of a $k$-uniform
$r$-regular hypergraph. Moreover, the following holds:

\begin{claim}\label{bipartity1}
    A function $f:V\rightarrow \{0,1\}$ is an indicator function
    of an $\ell$-fold transversal in a $k$-uniform $r$-regular
hypergraph
    $G$  if and only if $f$ is a perfect coloring of the multigraph
    $\mathcal{M}_{12}(\mathcal{D}(G))$  with the quotient matrix
$S=\begin{pmatrix}
\ell r & (k-\ell)r\\
\ell r & (k-\ell)r
\end{pmatrix}$.
\end{claim}
\begin{proof}
It is clear, that each $\ell$-fold transversal generates some
perfect $3$-coloring  of $\mathcal{D}(G)$: the first two colors are
the transversal and its complement in the first part of
$\mathcal{D}(G)$, all vertices of the second part of
$\mathcal{D}(G)$ are colored by the third color.  By Proposition
     \ref{bipartity},  the restriction of this coloring
     to the first part of $\mathcal{D}(G)$  is a perfect $2$-coloring
     of $\mathcal{M}_{12}(\mathcal{D}(G))$.
      The quotient matrix of this coloring is obtained from the following argument.
Each vertex of the $\ell$-fold transversal is adjacent to itself $r$
times and also $\ell-1$ times to other vertices of the $\ell$-fold
transversal in every of $r$ incident  hyperedges. Each vertex from
the complement of the $\ell$-fold transversal is adjacent to $\ell$
vertices in every of $r$ incident hyperedges.

 Suppose that  $f:V(\mathcal{M}_{12}(\mathcal{D}(G)))\rightarrow \{0,1\}$ is
 a perfect $2$-coloring with  quotient matrix $S$.
Define $h(x)=f(x)-|f^{-1}(1)|/|V(G)|$. By Corollary \ref{corol5} it
follows that $h$ is an eigenvector   with eigenvalue $0$ of the
adjacency matrix of $\mathcal{M}_{12}(\mathcal{D}(G))$. Then
$YY^*h=\bar0$. Consequently, it holds $(Y^*h,Y^*h)=(YY^*h,h)=0$ and
$Y^*h=\bar0$. The last equality implies that the sum of values of
$h$ in each hyperedge equals $0$. Since $f$ takes only two distinct
values, $-|f^{-1}(1)|/|V(G)|$ and $1-|f^{-1}(1)|/|V(G)|$, $h$ takes
every of these the same numbers of times in each hyperedge. Then $f$
satisfies the same property.
\end{proof}

\section{Independent Sets}

A subset of vertices in a graph is called an {\it independent }
whenever it does not contain adjacent vertices. The definition of an
independent set is the same for multigraphs without loops. It is
well known   the Delsarte--Hoffman upper bound for the cardinality
of an independent set in an $r$-regular graph (see, for instance,
\cite{GM} Theorem 2.4.1). This bound is equal to $\frac{-\theta
v}{r-\theta}$, where $v$ is the number of vertices in the graph and
$\theta$ is the minimal eigenvalue of the graph.

\begin{theorem}\label{thDH}
{ Let  $\theta$ be the minimal eigenvalue of an $r$-regular
multigraph $G$. If each vertices of  $A\subset V(G)$ is adjacent to
at most $t<r$
     vertices of $A$, then
     $|A|\leq\frac{(t-\theta)|V(G)|}{r-\theta}$. If in addition the cardinality of  $A$
     equals $\frac{(t-\theta)|V(G)|}{r-\theta}$, then the indicator
     function $f$ of $A$ is a perfect $2$-coloring of  $G$ with  quotient
     matrix $S=\begin{pmatrix}
t & r-t \\
t-\theta & r-t+\theta
\end{pmatrix}$. }
\end{theorem}
\begin{proof}
Assume that $G$ is connected, in opposite case we can consider each
connected component separately. Put $n=|V(G)|$. Since the adjacency
matrix $M$ of $G$ is symmetric, there exists an orthonormal basis
consisting of eigenvectors $\phi_i$ of $M$. The regularity and
connectedness of $G$ implies that  the maximal eigenvalue
$r=\theta_0$ corresponds to the unique basis vector
$\phi_0=\mathbf{1}/\sqrt{n}$. Consider the expansion of $f$ with
respect to this basis as $f=\sum_i\alpha_i\phi_i$. Then
\begin{equation}\label{PCe11}
t(f,f)\geq (Mf,f)=\sum_i\alpha^2_i\theta_i,
\end{equation} where $\theta_i$
is the eigenvalue corresponding to $\phi_i$. By the equality
$|A|=(f,f)= \sum_i\alpha^2_i$ we obtain that $\sum\limits_{i\neq
0}\alpha^2_i= (f,f)-\alpha^2_0=|A|- |A|^2/n$. Therefore,
(\ref{PCe11}) and the
minimality of $\theta$ imply that $t|A|\geq (|A|- |A|^2/n)\theta + r|A|^2/n.$\\
 Since
$r-\theta>0$ we obtain the inequality $\frac{|A|}{n}\leq
\frac{t-\theta}{r-\theta}$.

The equality $\frac{|A|}{n}=\frac{t-\theta}{r-\theta}$ holds only in
the case as  $f=\phi+\alpha_0\varphi_0$, where $\phi$ is the
eigenvector corresponding eigenvalue
 $\theta$. By Corollary \ref{corol5} we obtain that
 $f$ is a perfect coloring of $G$. We can find the entry $s_{21}$ of $S$
by the equation $|A|(r-t)=s_{21}(n-|A|)$, where the left and right
parts of the equation are the numbers of edges incident to vertices
of different colors.
\end{proof}

If we put  $t=0$ in Theorem \ref{thDH}, then we obtain the well
known Delsarte--Hoffman  bound on the cardinality of independent
sets. Moreover, Theorem \ref{thDH} implies that the indicator
function of each independent set attaining the Delsarte--Hoffman
bound is a perfect coloring. Conversely, it is easy to see  that the
first color of a coloring with  quotient matrix $\begin{pmatrix}
0 & r \\
-\theta & r+\theta
\end{pmatrix}$ is an independent set attaining  the Delsarte--Hoffman
bound.

\begin{claim}\label{exerDH}
If  $T$ is a transversal of a  $k$-uniform $r$-regular hypergraph
$G$, then $T$ is an independent set of the multigraph obtained from
$\mathcal{M}_{12}(\mathcal{D}(G))$ by removing all loops, and it
attains the Delsarte--Hoffman bound.
\end{claim}
\begin{proof}
Proposition \ref{bipartity1} implies that each transversal of  $G$
generates a perfect coloring of $\mathcal{M}_{12}(\mathcal{D}(G))$
with quotient matrix $S=\begin{pmatrix}
r & (k-1)r \\
r & (k-1)r
\end{pmatrix}$.

By Proposition \ref{eigenfun}, all eigenvalues of
 $\mathcal{M}_{12}(\mathcal{D}(G))$ are nonnegative. Every
 eigenvalue of a quotient matrix of a perfect coloring of the graph belongs
  to the spectrum  of the graph (see Corollary \ref{p:perf_color_properties}).
Thus the eigenvalue  $0$ of $S$ is the minimal eigenvalue of
  $\mathcal{M}_{12}(\mathcal{D}(G))$.
  Each vertex of  $\mathcal{M}_{12}(\mathcal{D}(G))$ is incident to
$r$ loops. Removing all loops, we obtain a $(k-1)r$-regular
multigraph
 $\Gamma$ with the minimal eigenvalue $-r$. By the definition of transversal,
  we have
 that   $T$ is an independent set and
  $|T|=|V(\Gamma)|/k=\frac{ r|V(\Gamma)|}{(k-1)r+r}$. Then it attains the
Delsart--Hoffman bound.
\end{proof}

A complete  subgraph of a graph is called a {\it clique}. It is easy
to see that each clique corresponds to an independent set in the
complement of the graph and vice versa. It is  known  (see
\cite{GM}) that  the cardinality of each clique in an arc-transitive
graph $r$-regular graph is at most $1-\frac{r}{\lambda}$, where
$\lambda$ is the minimal eigenvalue of the graph. A clique with
cardinality $1-\frac{r}{\lambda_{\min}}$ in an $r$-regular graph is
called the {\it Delsarte clique}.

Consider a $(k-1)r$-regular multigraph $\Gamma$ with the minimal
eigenvalue $-r$ obtained from $\mathcal{M}_{12}(\mathcal{D}(G))$ by
 removing all loops.  It is clear that each hyperedge  of $G$ corresponds
to a clique of size $k$ in $\Gamma$.  If  $\Gamma$ is a simple graph
then this clique turns out to be a Delsarte clique because of
$k=1-\frac{(k-1)r}{-r}$. It is proved in \cite{KMP}, Theorem 2(ii),
that the indicator function of  a subset of vertices of a
distance-regular graph is a perfect $2$-coloring if the  subset
intersects each Delsarte clique in the same number of vertices. If
$\Gamma$ is a distance-regular graph then this theorem coincides
with Proposition \ref{bipartity1}.

\section{Combinatorial Designs}

The {\it Johnson graph} $J(n,k)$ is a graph whose vertices are
binary $n$-tuples of weight $k$ and two vertices are joined by an
edge if the Hamming distance between the corresponding $n$-tuples
equals $2$.

A {\it combinatorial design} with parameters $t$-$(n,k,\lambda)$ is
a collection of $k$-elements subsets ({\it blocks}) of the
$n$-element set such that each $t$-element subset is included into
exactly $\lambda$ blocks from the collection.

The blocks of a $t$-$(n,k,\lambda)$-design can be represented as the
vertices of $J(n,k)$. In the case of $\lambda=1$,
$t$-$(n,k,\lambda)$-designs are usually denoted by   $S(t,k,n)$, if
in addition $t=k-1$, then they are called by {\it Steiner systems}
of order $n$, for instance, the Steiner triple system as $k=3$, the
Steiner quadruple system as $k=4$ and so on.

A combinatorial design with parameters $t$-$(n,k,\lambda)$ can be
regarded as a $\lambda$-fold transversal in the hypergraph
$G_{n,k,t}$ whose vertices are all possible $k$-element blocks and
hyperedges consist of blocks including a fixed  $t$-element set.

In the case $t=k-1$, the multigraph
$\mathcal{M}_{12}(\mathcal{D}(G_{n,k,t}))$ without loops coincides
with $J(n,k)$.
 Proposition
 \ref{exerDH} implies a well-known statement:

\begin{corollary}
The Steiner systems $S(k-1,k,n)$ are  maximal independent sets in
$J(n,k)$.
\end{corollary}

Denote $\mathcal{M}_{12}(\mathcal{D}(G_{n,k,t}))$ without loops  by
$(J(n,k))_{k-t}$.  In the case $0<t<k-1$, we obtain $(J(n,k))_{k-t}$
from
  $J(n,k)$ by joining vertices of  $J(n,k)$ corresponding to blocks
  including a
  joint $t$-element subset  by the suitable number of edges.

Proposition  \ref{bipartity1} implies the following statement:

\begin{corollary}\label{c:design21}
A set $D$ is a $t$-$(n,k,\lambda)$-design if and only if the
indicator function of $D$ is a perfect $2$-coloring of
$(J(n,k))_{k-t}$ with quotient matrix\\ $\left(\begin{array}{cc}
{k\choose t}(\lambda-1)& {k\choose t}({{n-k}\choose{k-t}}-\lambda+1) \\
{k\choose t}\lambda&  {k\choose t}({{n-k}\choose{k-t}}-\lambda)\\
\end{array} \right)$.
\end{corollary}

For the Steiner systems $S(k-1,k,n)$ Corollary \ref{c:design21} was
known before. By Proposition \ref{exerDH} we have

\begin{corollary}\label{c:design11}
Each $t$-$(n,k,1)$-design is a maximal independent set in
$(J(n,k))_{k-t}$.
\end{corollary}

It is shown in  \cite{Keev, Glock}, that $t$-$(n,k,\lambda)$-designs
exist for sufficiently large $n$ if their parameters meet the
well-known arithmetic conditions. This guarantees the existence of
perfect $2$-colorings with corresponding parameters.

A matrix $H$ of size $n\times n$ is called the {\it Hadamard matrix}
if it consists of $\pm 1$ and satisfies the equation $HH^*=nI_n$. By
multiplying rows and columns of $H$ by $-1$ we can reduce $H$ to the
Hadamard matrix whose first row and first column consist only of
$1$'s. Each other row or column of such Hadamard matrix contains
exactly equal numbers of $1$ and $-1$. It is well known that the
Hadamard matrices exist only if $n=4m$. Consider a matrix $A_m$ of
size $(4m+3)\times (4m+3)$ obtained by removing the first row and
the first column of the reduced Hadamard matrix and by changing
entries $-1$ by $0$. We can regard all rows of $A_m$ as binary
$(4m+3)$-tuples of weight $2m+1$. It is well known that

\begin{claim}\label{MHad}
The collection of rows of $A_m$ is a design with parameters
$2$-$(4m+3,2m+1,m)$.
\end{claim}

\begin{corollary}
The Hadamard matrices are equivalent to perfect $2$-colorings of the
corresponding multigraph  with certain quotient matrix.
\end{corollary}

Let $q$ be a prime power. Consider an $n$-dimensional vector  space
$X_q^n$ over the Galois field $X_q=GF(q)$. The {\it Grassmann graph}
is a graph whose vertices are $k$-dimensional subspaces of $X_q^n$
and two vertices are joined by an edge whenever the corresponding
subspaces are intersected by a $(k-1)$-dimensional subspace.

A $q$-{\it design} or {\it subspace design} with parameters
$t$-$(n,k,\lambda)_q$ is a such collection of  $k$-dimensional
vector subspaces in $X_q^n$ that each  $t$-dimensional subspace is
included into exactly   $\lambda$ $k$-dimensional subspaces from
this collection. The subspace designs are  $q$-analogs of
combinatorial designs. By a similar way a subspace design with
parameters $t$-$(n,k,\lambda)_q$ can be regarded as a $\lambda$-fold
transversal in the hypergraph $G^q_{n,k,t}$. Vertices of
$G^q_{n,k,t}$ are all possible $k$-dimensional subspaces in $X_q^n$
and hyperedges of $G^q_{n,k,t}$ are collections of subspaces which
include a fixed $t$-dimensional subspace.  In the case $t=k-1$, the
multigraph $\mathcal{M}_{12}(\mathcal{D}(G^q_{n,k,t}))$ without
loops coincides with  $J_q(n,k)$.

Similar to the case of combinatorial designs above,  we obtain the
next corollary of Proposition \ref{bipartity1}.

\begin{corollary}\label{c:design23}
A set $D$ is a subspace design with parameters $t$-$(n,k,\lambda)_q$
if and only if the indicator function of $D$ is a perfect
$2$-coloring of $\mathcal{M}_{12}(\mathcal{D}(G^q_{n,k,t}))$ with
quotient matrix $\left(\begin{array}{cc}
{k\brack t}_q(\lambda-1)& {k\brack t}_q({{n-k}\brack{k-t}}_q-\lambda+1) \\
{k\brack t}_q\lambda&  {k\brack t}_q({{n-k}\brack{k-t}}_q-\lambda)\\
\end{array} \right)$, where ${k\brack
t}_q=\frac{(q^k-1)\cdots(q^{k-t+1}-1)}{(q^t-1)\cdots(q-1)}$ are the
Gaussian binomial coefficients.
\end{corollary}

{\it Spreads} are subspace designs  with parameters $1$-$(n,k,1)_q$.
Constructions of spreads and subspace designs with $\lambda>1$ are
studied for a long time. The first subspace design with  $t>1$ and
$\lambda=1$ was found recently. This is the subspace design with
parameters $2$-$(13,3,1)_2$, in other words, this is a $2$-analog of
the Steiner triple system \cite{BEW}.

\section{Difference Sets and Bent Functions}

Let $K$ be a finite abelian group. A  {\it partial difference set}
     with parameters $(v,k,\lambda,\mu)$ is the set $D\subseteq K$,
    $|K|=v$, $|D|=k$, where for each nonzero $ a\in D$ there
    exist exactly $\lambda$ pairs $d_1, d_2\in D$ such that
    $d_1-d_2=a$, and  for each nonzero  $ a\in K\setminus D$ there
    exist exactly
     $\mu$ pairs $d_1, d_2\in D$ such that
    $d_1-d_2=a$.
In the case $\mu=\lambda$, every partial difference set  $D\subseteq
K$ is called simply a {\it difference set}
     with parameters $(v,k,\lambda)$.

A {\it convolution} of functions $f,g:K\rightarrow \mathbb{R}$ is
defined by the equation  $f*g(y)=\sum\limits_{x\in K}f(x)g(y-x)$.
Denote by $\delta:K\rightarrow \mathbb{R}$ the function taking the
value  $|K|$ in the  identity element of $K$ and $0$ in the other
elements. Let $\mathbf{1}$ be the function that takes value $1$
identically.
 It is known  the following characterization of
 partial difference sets, see \cite{Ma} for instance.

\begin{claim}\label{c:diffset}
A set $D\subseteq K$ is a partial difference set with parameters
$(v,k,\lambda,\mu)$ if and only if
 $\chi_{_D}
*\chi_{_{-D}}=\lambda\cdot\chi_{_D}+\mu\cdot(\mathbf{1}-\chi_{_D})+
(k-\mu)\delta$.
\end{claim}

Let $D$ be a difference set. Consider the matrix
$Q_{_D}(x,y)=\chi_{_D}(x-y)$.
 By the definition, every pair of columns and every pair of rows of $Q_{_D}$
 contains the same number of pairs $(1,1)$ as well as the other pairs
  $(0,1),(1,0),(0,0)$. If we regard the rows of $Q_{_D}$ as an
  indicator function of blocks, then the collection of the rows is a
   $2$-$(v,k,\lambda)$-design. If, as in this case, the cardinality of
    $2$-design is equal to its
    length, then the design is called {\it symmetric}.

As a definition of {\it Boolean bent function} $b:\{0,1\}^n
\rightarrow \{0,1\}$ we can take the following property: the
convolution $(-1)^b*(-1)^b$ takes the value  $2^n$ in the zero
vector and zeros in the remaining arguments. It is clear that
$(-1)^b=\mathbf{1}-2b$. By Proposition \ref{c:diffset} we obtain
that  bent functions correspond to  difference sets in
$\mathbb{Z}_2^n$. It is known (see \cite{Mes0} for instance) that
bent functions exist if and only if $n$ is even and the number of
values $1$ of a bent function is equal to either $2^{n-1}+2^{n/2-1}$
or $2^{n-1}-2^{n/2-1}$. Moreover, the following statement holds.

\begin{claim}\label{cor:bent_diffset} Let $b$ be a Boolean bent function and let
 $B=\{x\in\{0,1\}^n\mid b(x)=1 \}$, i.\,e., $b=\chi_{_B}$. Then $B$ is a
 difference set with parameters either
$(2^n,2^{n-1}-2^{n/2-1},2^{n-2}-2^{n/2-1})$ or
$(2^n,2^{n-1}+2^{n/2-1},2^{n-2}+2^{n/2-1})$. The converse statement
is also true.
\end{claim}

The parameters of the difference set in Proposition
\ref{cor:bent_diffset} is called the {\it McFarland parameters}
\cite{MF}. It is easy to prove that each difference set with the
McFarland parameters is possible to transform into some Hadamard
matrix.

A connected graph $G$ is  {\it strongly regular} with parameters
$(v,k,\lambda,\mu)$, whatever $|V(G)|=v$, the degree of each vertex
equals  $k$, each two adjacent vertices have $\lambda$ common
neighbors, and each two non-adjacent vertices have $\mu$ common
neighbors.
For a group $K$ we consider a  set  $A\subset K$ of generator such
that
 $A^{-1}=A$ and $\varepsilon\notin A$, where $\varepsilon$ is the identity element of $K$.
The {\it Cayley graph} $Cay(K,A)$ is the graph whose vertices are
the elements of $K$ and two vertices  $x,y\in K$ are joined by an
edge if $y=xa$ for some $a\in A$. The conditions $A^{-1}=A$ and
$\varepsilon\notin A$ imply that $Cay(K,A)$ is a simple graph.

Consider the Cayley graph of an abelian group $K$ with  a set $D$ of
generators such that $D=-D$ and $0\not \in D$. The definition of a
partial difference set with parameters  $(v,k,\lambda,\mu)$ implies
that the graph $Cay(K,D)$ is  strongly regular  with parameters
$(v,k,\lambda,\mu)$. Indeed, vertices $x,y\in K$ are adjacent in
$Cay(K,D)$ if and only if  $x+a=y$ for some  $a\in D$. Thus, by the
definition of partial difference set, there exist exactly $\lambda$
pairs $d_1,d_2\in D$ such that  $a=d_1-d_2$. Consequently, there
exist exactly $\lambda$  vertices  $b$, $b=x+d_1=y+d_2$ adjacent to
both vertices $x$ and $y$ in $Cay(K,D)$. By a similar way we can
consider the case of non-adjacent vertices $x,y\in K$. It is easy to
see that the converse statement is true, i.\,e.,   if  $Cay(K,D)$ is
a strongly regular graph then $D$ is a partial difference set in
 $K$.

As we showed above, each difference set $D$
 with parameters $(v,k,\lambda)$ corresponds to a symmetric  $2$-$(v,k,\lambda)$-design.
 If $D=-D$ and $0\not \in
D$ then  $Q_{_D}(x,y)=\chi_{_D}(x-y)$ is the adjacency matrix of
$Cay(K,D)$. The converse statement is also true: if a matrix $A$ is
symmetric and has only zeros on the main diagonal, and rows of $A$
are blocks of a $2$-$(v,k,\lambda)$-design, then $A$ is the
adjacency matrix of some strongly regular graph with parameters
$(v,k,\lambda,\lambda)$.

\section{Difference Sets, Strongly Regular Graphs and Bent Functions as Perfect Colorings
}

Further  we will show that each strongly regular graph and, in
particular, each partial difference set corresponds to a perfect
$2$-coloring of some hypergraph.

Consider the complete graph $K_n$ on $n$ vertices. Define a
hypergraph $\Gamma_n$ whose vertices are edges of $K_n$ and a triple
of vertices constitutes a hyperedge wherever these vertices generate
a triangle in $K_n$. Let $G$ be a graph on the same set of vertices.
Define a coloring  of $\Gamma_n$ as follows: $f(e)=1$ if $e\in E(G)$
and $f(e)=0$ if $e\not \in E(G)$. The coloring $f$ is perfect if and
only if $G$ is strongly regular. In order to see this we need to
verify the property: for each two vertices $u$ and $v$ the numbers
of vertices adjacent  to either $u$ or $v$, adjacent to both  $u$
and $v$, and adjacent to neither $u$ no $v$ in $G$ depend only on
the adjacency of $u$ and $v$ in $G$. This condition coincides with
the definition of a strongly regular graph. In particular, for the
strongly regular graph $G$ with parameters $(n,k,\lambda,\mu)$ we
obtain that if $u$ and $v$ are adjacent in $G$, then the number of
vertices which are adjacent to both vertices $u$ and $v$ equals
$\lambda$; the number of vertices which are adjacent to either $u$
or $v$ equals $k-\lambda$; the number of vertices which are not
adjacent to $u$ or $v$  equals $v-2-2k+\lambda$.

Next we will consider the abelian group $\mathbb{Z}_2^n$. Denote by
$\Delta_n$ a $3$-uniform hypergraph whose vertices are elements of
$\mathbb{Z}_2^n$ without $0$ and  triples $\{a_1, a_2, a_3\}$ are
hyperedges wherever $a_1+a_2+a_3=0$. Notice that the equality
$a_1+a_2+a_3=0$ implies that all three elements $a_i\neq 0$,
$i=1,2,3$ are different. If $D$ is a partial difference set
$D\subset \mathbb{Z}_2^n$ with parameters $(v,k,\lambda,\mu)$,
 then $\chi_{_{D}}$ is a perfect $2$-coloring of
$\Delta_n$. Indeed,  the equation $a_1+a_2+a_3=0$ is equivalent to
the equation $a_1=a_2-a_3$ in $\mathbb{Z}_2^n$. Then each vertex
$a_1\in D$ is incident to $\lambda$ hyperedges containing three
elements of $D$ and each vertex $a_1\not\in D$ is incident to $\mu$
hyperedges containing two elements of $D$. Since every pair $a,b\in
\mathbb{Z}_2^n\setminus \{0\}$ belongs to exactly one hyperedge, we
can calculate the numbers of hyperedges containing a vertex $a\in D$
 or $a\not \in D$ for all  color compositions of hyperedges. These numbers depend
 only on the color of the vertex. Thus,
 $\chi_{_{D}}$ is a perfect $2$-coloring by the definition.
 The converse statement is also true. Let $D$ be the set of $1$ colored vertices.
  It is sufficient to observe
 that the parameter $\lambda$ of the partial difference set $D$ is equal
 to the number of hyperedges of the color composition $(1,1,1)$
 which contain a fixed vertex of color $1$, and the parameter $\mu$  is equal
 to the number of hyperedges of the color composition $(0,1,1)$
 which contain a fixed vertex of color $0$.

Therefore, the next statement follows from Proposition
\ref{cor:bent_diffset}.

\begin{claim}\label{avgdop}
Boolean bent functions one-to-one correspond to perfect
$2$-colorings of $\Delta_n$ with certain quotient matrix.
\end{claim}

 By Proposition \ref{cl:edge_coloring}, a perfect
$2$-coloring of $\Delta_n$ induces a perfect coloring of the line
graph $\mathcal{E}(\Delta_n)$. Since an intersection of every two
hyperedges of $\Delta_n$ consists of at most one vertex,
$\mathcal{E}(\Delta_n)$ is a simple graph. Notice that a
$2$-coloring $\chi_{_{D}}$ of $\Delta_n$ induces a $4$-coloring of
$\mathcal{E}(\Delta_n)$ because each hyperedge consists of three
vertices and, consequently,  it can  contain
 $0$,
$1$, $2$ or $3$ vertices of the first color. Consider  elements of
$\mathbb{Z}_2^n$ as elements of the vector space $X_2^n$. It is easy
to see that hyperedges of $\Delta_n$ one-to-one correspond to
$2$-dimensional subspaces of $X_2^n$. Furthermore, two hyperedges of
$\Delta_n$ have common vertex if and only if the corresponding
subspaces meet along a $1$-dimensional subspace. Thus the graph
$\mathcal{E}(\Delta_n)$ is equivalent to the Grassmann graph
$J_2(n,2)$ by the definition.

By Proposition  \ref{cor:bent_diffset}, we obtain that a Boolean
bent function generates a difference set in $\mathbb{Z}_2^n$. As
proved above, this set induces a perfect $4$-coloring of $J_2(n,2)$.
In the following theorem we prove that the converse holds: each
perfect coloring of $J_2(n,2)$ with certain quotient matrix
determines a Boolean bent function. We consider only the case as
bent functions with $2^{n-1}+2^{n/2-1}$ values $1$, the case as bent
functions with $2^{n-2}-2^{n/2-1}$ values $1$ is  similar.

\begin{theorem}\label{avg}
Bent functions $b:\{0,1\}^n\rightarrow \{0,1\}$ with
$|\supp(b)|=2^{n-1}+2^{n/2-1}$ and $b(\bar 0)=1$ one-to-one
correspond to perfect colorings of
 $J_2(n,2)$ with quotient matrix\\
$\begin{pmatrix}
3(2^{n-3}-2^{\frac{n}{2}-2}-1)& 3\cdot2^{n-2}-3 & 3(2^{n-3}+2^{\frac{n}{2}-2}) & 0\\
2^{n-2}-2^{\frac{n}{2}-1} & 5\cdot2^{n-3}-2^{\frac{n}{2}-2}-5& 2^{n-1}+2^{\frac{n}{2}-1} & 2^{n-3}+2^{\frac{n}{2}-2}-1\\
2^{n-3}-2^{\frac{n}{2}-2} & 2^{n-1}-2^{\frac{n}{2}-1}-1 & 5\cdot2^{n-3}+2^{\frac{n}{2}-2}-3 & 2^{n-2}+2^{\frac{n}{2}-1}-2 \\
0 & 3(2^{n-3}-2^{\frac{n}{2}-2}) & 3\cdot2^{n-2} & 3(2^{n-3}+2^{\frac{n}{2}-2}-2)\\
\end{pmatrix}$.
\end{theorem}
\begin{proof}
At first we obtain a  coloring  of $J_2(n,2)$ from a given bent
function. We determine a   vertex color of $J_2(n,2)$ as the number
of values $1$ of the bent function on the corresponding
$2$-dimensional subspace. Since $b(\bar 0)=1$, there exit the $1$st,
$2$nd, $3$rd and $4$th colors. We need to prove that this coloring
is perfect.

 By
the definition of bent function it follows that
\begin{equation}\label{eqavg1}
|\{x\in \{0,1\}^n:b(x)\oplus b(x+y)=0\}|=|\{x\in
\{0,1\}^n:b(x)\oplus b(x+y)=1\}|
\end{equation}
for each $y\neq \bar0$.

Consider two $2$-dimensional subspaces meeting along $\{\bar0, y\}$.
Denote by $A_{\alpha\beta}$ the number of pairs of values
$(\alpha,\beta)$ among all pairs $(b(x),b(x+y))$. Then
$A_{01}+A_{10}=A_{00}+A_{11}$ by (\ref{eqavg1}). Moreover,
$A_{01}+A_{10}+A_{00}+A_{11}=2^n$, $A_{01}=A_{10}$ from the symmetry
and $A_{10}+A_{11}=|\supp(b)|=2^{n-1}+2^{n/2-1}$. Hence, it is easy
to calculate that
 $A_{01}=A_{10}=2^{n-2}$, $A_{11}=2^{n-2}+2^{\frac{n}{2}}$ and
$A_{00}=2^{n-2}-2^{\frac{n}{2}-1}$.

If a color of some subspace is known then we can calculate the
number of subspaces  of each color adjacent to it in $J_2(n,2)$. For
instance, the $4$th color subspace is not adjacent to subspaces of
the first color; it is  adjacent to  $3A_{00}/2$ subspaces of $2$nd
color because each pair of vectors  $(x,x+y)$,
$(b(x),b(x+y))=(0,0)$, is counted twice $(x,x+y)$ and $(x+y,x)$; it
is  adjacent to $(3A_{11}-12)/2$ subspaces of the $4$th color
because the initial subspace contains four vectors of the value $1$
and by fixing every of three nonzero vectors we have $(A_{11}-4)/2$
adjacent subspaces of the $4$th color. Then the coloring of
$J_2(n,2)$ is perfect by the definition. Furthermore, we can
calculate the quotient matrix of the coloring

 $\frac12\begin{pmatrix}
3A_{00}-6& 6A_{01}-6 & 3A_{11} & 0\\
2A_{00} & 4A_{01}+A_{00}-10& 2A_{01}+2A_{11} & A_{11}-2\\
A_{00} & 2A_{01}+2A_{00}-2 & 4A_{10}+A_{11}-6 & 2A_{11} -4\\
0 & 3A_{00} & 6A_{01} & 3A_{11}-12 \\
\end{pmatrix}$.

Suppose that a perfect $4$-coloring of $J_2(n,2)$  with this
quotient matrix is given. Denote $b$ as follows: let $b(\bar0)=1$
and $b(x)=1$ at all points $x\in X_2^n$ belonging to subspaces of
the $4$th color. At the remaining  points $x\in X_2^n$ we put
$b(x)=0$. At the last part of the proof we verify that the resulting
$b$ is a bent function.

Denote by $M(x)$ the number of the $4$th color subspaces containing
a vector $x\in B=\{x\neq \bar 0 :b(x)=1\}$. By the entry ($1$st row,
$4$th column) of the quotient matrix we conclude that subspaces of
the $1$st color  do not contain $x$. Moreover, all vectors of a
subspace belong to  $B$ if and only if it is the $4$th color
subspace   because subspaces of other colors are adjacent to
subspaces of the $1$st color. By the entry ($4$th row, $4$th column)
of the quotient matrix we conclude that the average of $M(x)$ over
an arbitrary  subspace of the $4$th color equals
$\frac{A_{11}}{2}-1$, where $A_{11}=2^{n-2}+2^{\frac{n}{2}}$.
Suppose that there is a such $x\in B$ that $M(x)\neq
\frac{A_{11}}{2}-1$. Let us take $x\in B$ such that the difference
$|M(x)-(\frac{A_{11}}{2}-1)|$ attains the maximum.

Suppose $M(x)> \frac{A_{11}}{2}-1$ and this difference is greater
than the difference $|M(u)-(\frac{A_{11}}{2}-1)|$ for each $u\in B$
such that  $M(u)< \frac{A_{11}}{2}-1$. We will show that there exist
one more vector $v\in B$ such that $M(v)\geq A_{11}/2$. Consider
some subspace containing $x$. For both vectors $z\neq x$ and $x+z$
we have that $M(z)< \frac{A_{11}}{2}-1$ and
$M(x+z)<\frac{A_{11}}{2}-1$ because the average value of $M(y)$ over
every subspace of the $4$th color is $\frac{A_{11}}{2}-1$ but the
difference between $M(x)$ and $\frac{A_{11}}{2}-1$ is strongly
maximal by the assumption. If $M(z)>1$ or $M(x+z)>1$ than there
exist at least two distinct subspaces, and each of them contains a
vector $v\in B$ such that $M(v)\geq A_{11}/2$. If $M(z)=1$ and
$M(x+z)=1$ then $M(x)= \frac{3A_{11}}{2}-6$ and each subspace
containing $x$ is adjacent to $\frac{3A_{11}}{2}-6$ subspaces of the
$4$th color. It contradicts the definition of the quotient matrix.

We have shown  that there exist at least two vectors $x,y\in B$ such
that $M(x)\geq M(y)\geq A_{11}/2$. Consider the subspace $\{\bar 0,
x, y, x+y\}$. We obtain that $x+y\not \in B$ because $x+y\in B$
contradicts the choice of $x$. But there are no subspaces of colors
$1$, $2$ or $3$ adjacent to $M(x)+M(y)-1\geq  A_{11}-1$ subspaces of
the $4$th color.

Next we suppose that the difference $|M(x)-(\frac{A_{11}}{2}-1)|$ is
maximal (not necessarily strictly maximal) for some $x\in B$ such
that $M(x)< \frac{A_{11}}{2}-1$. Consider some $4$th color subspace
containing $x$. In this subspace there exists a vector $z\in B$ such
that $M(z)> \frac{A_{11}}{2}-1$ because the average number of $M(v)$
over every subspace of the $4$th color is the constant. By a similar
way we conclude that every subspaces (among $M(z)$ $4$th color
subspaces) containing $z$ contains vectors $y\in B$ such that $M(y)<
\frac{A_{11}}{2}-1$. We will use three distinct $x,y,y'\in B$
satisfied the last inequality.

Consider the subspace $\Omega=\{\bar 0, x, y, x+y\}$. This subspace
is adjacent to $M(x)+M(y)$ subspaces of color $4$ because the
opposite assumption $x+y \in B$ contradicts maximality of the
difference $|M(x)-(\frac{A_{11}}{2}-1)|$. The color of $\Omega$ does
not equal
 $4$ because of $x+y \not \in B$ and it does not equal  $1$ because
 the
$1$st color subspaces are not adjacent to subspaces of color $4$.
The color of $\Omega$ does not equal $3$ because
$M(x)+M(y)<A_{11}-2$, i.\,e., the $3$rd color of $\Omega$
contradicts the quotient matrix. It leaves the last possibility that
the color of $\Omega$ is equal to $2$. By the quotient matrix,
$\Omega$ is adjacent to $A_{00}=2^{n-2}-2^{\frac{n}{2}-1}$ subspaces
of the $1$st color. These subspaces do not contain $x$ or $y$
because subspaces of colors $1$ and $4$ are not adjacent.

Consider subspace $\Omega'=\{\bar 0, x, y', x+y'\}$. By the same way
we obtain that $x+y'$ belongs to $A_{00}$ subspaces of color $1$.
Then  the subspace  $\{\bar 0, x+y, x+y', y+y'\}$ is adjacent to
$2A_{00}-1$ subspaces of color $1$ at least. By the quotient matrix
such subspace does not exist. By this contradiction we proved that
 $M(x)=\frac{A_{11}}{2}-1$ for each $x\in B$.

By the $4$th column of the quotient matrix we obtain that every
subspace of color $i$ contains $i-1$ vectors of $B$. Consequently,
the coloring of $J_2(n,2)$ induces a perfect $2$-coloring of
$\Delta_n$. It is easy to see that parameters of the coloring of
$\Delta_n$ correspond to parameters of a coloring induced by a bent
function. Therefore, the coloring of $J_2(n,2)$ with the quotient
matrix determined above  induces a bent function by Proposition
\ref{avgdop}.
\end{proof}

It is not difficult to verify that, combining pairwise even and odd
colors in the above $4$-coloring of $J_2(n,2)$, we obtain a perfect
coloring with quotient matrix

 $\begin{pmatrix}
3\cdot 2^{n-2}-3& 3\cdot 2^{n-2}-3 \\
3\cdot 2^{n-2} & 3\cdot 2^{n-2}-6\\
\end{pmatrix}$.

But these perfect $2$-colorings can not one-to-one correspond to
bent functions. When we add an affine function to a bent function,
we will get a new bent function but the $2$-coloring will remain the
same.

\end{document}